\documentclass[a4paper]{amsart}
\usepackage[all]{xy}
\usepackage{ifthen}
\usepackage{amssymb}

\SelectTips{eu}{}
\calclayout

\title[An informal introduction to dg categories]{An informal introduction to dg categories}

\thanks{\today}

\author[Xiaofa Chen, Xiao-Wu Chen]{Xiaofa Chen, Xiao-Wu Chen}
\address{Key Laboratory of Wu Wen-Tsun Mathematics, Chinese Academy of Sciences
\\ School of Mathematical Sciences, University of Science and Technology of China, Hefei 230026, Anhui, PR China}
\email{cxf2011@mail.ustc.edu.cn, xwchen@mail.ustc.edu.cn}

\theoremstyle{plain}
\newtheorem{lem}{Lemma}[subsection]
\newtheorem{prop}[lem]{Proposition}
\newtheorem{cor}[lem]{Corollary}
\newtheorem{thm}[lem]{Theorem}

\theoremstyle{remark}

\theoremstyle{definition}
\newtheorem{rem}[lem]{Remark}
\newtheorem{exm}[lem]{Example}
\newtheorem{defn}[lem]{Definition}

\numberwithin{equation}{subsection}

\makeindex

\begin{document}

\begin{abstract}
In this informal introduction to dg categories,  the slogan is that dg categories are more rudimentary than triangulated categories. We recall some details on the dg quotient category introduced by Bernhard Keller and Vladimir Drinfeld.
\end{abstract}

\maketitle

\setcounter{tocdepth}{2}
\tableofcontents

\section{Introduction}\label{se:intro}

Why dg categories?  There are two levels of reasons!

\vskip 10pt

\noindent {\bf Level One.}
\vskip 10pt

\begin{enumerate}
\item[$\bullet$] a dg category = a dg algebra with several objects
\item[$\bullet$] the dg endomorphism algebras of certain complexes arise naturally in the derived Morita theory and Koszul duality
\item[$\bullet$] each triangulated category appearing naturally in algebra is essentially a triangulated subcategory of $\mathbf{D}(\mathcal{A})$, the derived category of dg modules over a dg category $\mathcal{A}$
\end{enumerate}

\vskip 10pt

\noindent {\bf Level Two.}
\vskip 10pt

\begin{enumerate}
\item[$\bullet$] the non-functoriality of cones in a triangulated category causes many problems, e.g.,  tensor products and functor categories between triangulated categories are not triangulated
\item[$\bullet$] any exact (= strongly pretriangulated) dg category $\mathcal{A}$ has functorial cones, and gives rise to a triangulated category $H^0(\mathcal{A})$, by taking cohomologies
\item[$\bullet$] tensor products and functor categories of dg categories are naturally dg
\item[$\bullet$] the homotopy category $\mathbf{Hodgcat}$ of small dg categories provides a suitable framework
\end{enumerate}

\vskip 10pt

For algebraists, the slogan is

$${\boxed{\bf \mbox{dg categories are more rudimentary  than triangulated categories!}}}$$

\vskip 10pt

The main references are \cite{Ke94, Dri, Ke06}. This note might serve as a preparation to read the survey \cite{Ke06}. The dg quotient category is introduced in \cite{Ke99, Dri}. We work over a commutative ring $k$.

\section{Preliminaries}

In this section, we recall basic facts on dg categories.

\subsection{The definitions and examples}

\begin{defn}
A \emph{dg category} $\mathcal{A}$ is a category subject to
\begin{enumerate}
\item[$\bullet$] ${\rm obj}(\mathcal{A})$: usually denoted by $x, y, z \cdots$
\item[$\bullet$] each Hom set $\mathcal{A}(x, y)$ is a complex, denoted by
$$\mathcal{A}(x, y)=(\bigoplus_{p\in \mathbb{Z}} \mathcal{A}(x, y)^p, d_\mathcal{A}).$$
The element $f\in \mathcal{A}(x, y)^p$ is homogeneous of degree $p$ with the notation $|f|=p$,
\end{enumerate}
such that the composition
$$\mathcal{A}(y, z)\otimes \mathcal{A}(x, y)\longrightarrow  \mathcal{A}(x, z)$$
is a chain map. In more details, for any homogeneous morphisms $f\colon x\rightarrow y$ and $g\colon y\rightarrow z$, we have
$$|g\circ f|=|g|+|f| \mbox{ and } d_\mathcal{A}(g\circ f)=d_\mathcal{A}(g)\circ f +(-1)^{|g|} g\circ d_\mathcal{A}(f).$$
\end{defn}

\begin{rem}
\begin{enumerate}
\item A nice exercise: each identity endomorphism $1_x$ is homogeneous of degree zero satisfying $d_\mathcal{A}(1_x)=0$.
    \item The complex $\mathcal{A}(x, y)$ is sometimes denoted by ${\rm Hom}_\mathcal{A}(x, y)$, the Hom-complex. \end{enumerate}
\end{rem}

In a dg category $\mathcal{A}$, a homogeneous morphism $f$ is \emph{closed} provided that $d_\mathcal{A}(f)=0$; a \emph{dg isomorphism} means a closed isomorphism of degree zero.

  We will later omit the word ``homogeneous" by the following convention.

$${\boxed{\bf \mbox{In the dg setting, we only consider homogeneous morphisms and elements!}}}$$

\vskip 10pt

\begin{defn}
For a dg category $\mathcal{A}$, its \emph{ordinary category} $Z^0(\mathcal{A})$ is defined as follows: the same objects with $\mathcal{A}$, and morphisms are given by closed morphisms of degree zero.

 Similarly, the \emph{homotopy category} $H^0(\mathcal{A})$ has the same objects whose morphism spaces are given by the zeroth cohomologies of the Hom-complexes in $\mathcal{A}$.
\end{defn}

\begin{rem}
\begin{enumerate}
\item $H^0(\mathcal{A})$ is a factor category of $Z^0(\mathcal{A})$.
\item The graded version: $Z(\mathcal{A})$ and $H(\mathcal{A})$.
\end{enumerate}
\end{rem}

\begin{defn}
For two dg categories $\mathcal{A}$ and $\mathcal{B}$, a \emph{dg functor} $F\colon \mathcal{A}\rightarrow \mathcal{B}$ is a $k$-linear functor such that
$$F_{x, y}\colon \mathcal{A}(x, y)\longrightarrow \mathcal{B}(Fx, Fy)$$
is a chain map. In other words, $|F(f)|=|f|$ and $F(d_\mathcal{A}(f))=d_\mathcal{B}(Ff)$.
\end{defn}

We might define the notion of a \emph{dg equivalence},  and prove the following theorem: a dg functor is a dg equivalence if and only if it is fully faithful and dg dense (i.e. every object in the target category is dg isomorphic to some object in the image). However,  the more useful notions are as follows.

\begin{defn}
A dg functor $F\colon \mathcal{A}\rightarrow \mathcal{B}$ is \emph{quasi-fully faithful}, if each $F_{x, y}$ is a quasi-isomorphism. If in addition $H^0(F)\colon H^0(\mathcal{A})\rightarrow H^0(\mathcal{B})$ is dense, $F$ is called a \emph{quasi-equivalence}.
\end{defn}

\begin{exm}\label{exm:dg}
\begin{enumerate}
\item A dg algebra = a dg category with a single object. Then dg functors between these dg categories are just dg algebra homomorphisms.

    Accordingly, for each object $x\in \mathcal{A}$, $\mathcal{A}(x, x)$ is a dg algebra, the \emph{dg endomorphism algebra} of $x$.
\item For each $n\in \mathbb{Z}$, the $n$-th \emph{disc} $\mathcal{D}(n)$ is defined as follows: the object set is $\{x, y\}$ such that $\mathcal{D}(n)(x, x)=k1_x$, $\mathcal{D}(n)(y, y)=k1_y$, $\mathcal{D}(n)(y, x)=0$ and $\mathcal{D}(n)(x, y)=k\delta\oplus k\epsilon$, with $|\delta|=-n$ and $|\epsilon|=-n+1$. The differential is determined by $d(\delta)=\epsilon$.

    The $(n-1)$-th \emph{sphere} dg category $\mathcal{S}(n-1)$ is the non-full subcategory of $\mathcal{D}(n)$ with the same object set and spanned by $\{1_x, 1_y, \epsilon\}$.

\item The \emph{opposite} dg category $\mathcal{A}^{\rm op}$: $a\circ^{\rm op} b=(-1)^{|a|\cdot |b|} b\circ a$.

\item The dg category $C_{\rm dg}(k)$ of $k$-modules: the objects are complexes $V=(\bigoplus_{n\in \mathbb{Z}} V^n, d_V)$ of $k$-modules; the $p$-th component of the Hom-complex ${\rm Hom}(V, W)$ is given by
    \begin{align*}
    {\rm Hom}(V, W)^p &=\prod_{n\in \mathbb{Z}} {\rm Hom}_k(V^n, W^{n+p})\\
    &=\{f\colon V\rightarrow W \; |\; f \mbox{ is homogeneous of degree } p\}.\end{align*}
    The differential is given by
    $$d(f)=d_W\circ f-(-1)^{|f|} f\circ d_V.$$
\end{enumerate}
\end{exm}

\begin{rem}
\begin{enumerate}
\item In Example \ref{exm:dg} (4),  the dg algebra ${\rm End}(V)={\rm Hom}(V, V)$ is the well-known dg endomorphism algebra of the complex $V$!

\item Fact: $Z^0(C_{\rm dg}(k))$ coincides with the abelian category $\mathbf{C}(k\mbox{-Mod})$ of cochain complexes, and $H^0(C_{\rm dg}(k))$ is the classical homotopy category $\mathbf{K}(k\mbox{-Mod})$!
    \end{enumerate}
\end{rem}

\subsection{Tensor products and functor categories}

The \emph{tensor product} $\mathcal{A}\otimes \mathcal{B}$ of two dg categories: the object class is ${\rm obj}(\mathcal{A})\times {\rm obj}(\mathcal{B})$, whose element $(x, y)$ is denoted by $x\otimes y$,   and the Hom-complex is
$$\mathcal{A}\otimes \mathcal{B}(x\otimes y, x'\otimes y')=\mathcal{A}(x, x')\otimes \mathcal{B}(y, y').$$

The Koszul sign rule:
$$(a\otimes b) \circ (a'\otimes b')=(-1)^{|b|\cdot |a'|} a\circ a'\otimes b \circ b' $$
and
$$d(a\otimes b)=d_\mathcal{A}(a)\otimes b + (-1)^{|a|} a\otimes d_\mathcal{B}(b).$$

\begin{rem}
The commutativity holds:
$$(a\otimes 1_{y'}) \circ (1_x\otimes b)=a\otimes b= (-1)^{|a|\cdot |b|} (1_{x'} \otimes b) \circ (a\otimes 1_{y})$$
for any $a\colon x\rightarrow x'$ in $\mathcal{A}$ and $b\colon y\rightarrow y'$ in $\mathcal{B}$.
\end{rem}

The \emph{functor dg category} $\mathcal{H}om(\mathcal{A}, \mathcal{B})$: the objects are dg functors from $\mathcal{A}$ to $\mathcal{B}$; a morphism $\eta\colon F\rightarrow G$ of degree $p$ between two such functors is given by a family  $\eta=(\eta_x)_{x\in {\rm obj}(\mathcal{A})}$ of morphisms $\eta_x \colon F(x)\rightarrow G(x)$ of degree $p$ in $\mathcal{B}$ satisfying
$$G(a)\circ \eta_x =(-1)^{|\eta|\cdot |a|} \eta_{x'}\circ F(a)$$
for all $a\colon x\rightarrow x'$.  These morphisms form  the $p$-th component of the Hom-complex ${\rm Hom}(F, G)$. The differential is given by
$$(d\eta)_x=d_\mathcal{B}(\eta_x).$$

\begin{lem}
For any dg categories $\mathcal{A}, \mathcal{B}$ and $\mathcal{C}$, there is a canonical isomorphism of dg categories
$$\mathcal{H}om(\mathcal{A}\otimes \mathcal{B}, \mathcal{C})\longrightarrow \mathcal{H}om(\mathcal{A}, \mathcal{H}om(\mathcal{B}, \mathcal{C})), \quad F\mapsto (x\mapsto F(x\otimes -)).$$
\end{lem}

\begin{exm}
Let $A$ and $B$ be two dg algebras, viewed as dg categories. Then $A\otimes B$ is just the tensor product of dg algebras. The objects in the dg categories $\mathcal{H}{\emph om}(A, B)$ are just dg algebra homomorphisms from $A$ to $B$. For two homomorphisms $\theta, \theta'\colon A\rightarrow B$ and $p\in \mathbb{Z}$, we have
$${\rm Hom}(\theta, \theta')^p=\{b\in B^p\; |\; \theta'(a) b=(-1)^{|a|\cdot |b|}b\theta(a) \mbox{ for all } a\in A\}.$$
In particular, ${\rm Hom}(\theta, \theta')$ is a subcomplex of $B$.
\end{exm}

The \emph{morphism dg  category} ${\emph mor}(\mathcal{A})$ of $\mathcal{A}$ is introduced in \cite[Subsection 2.9]{Dri} : the objects are triples $(x, y; a)$, with $x, y\in {\rm obj}(\mathcal{A})$ and $a\colon x\rightarrow y$ a closed morphism of degree zero; the $p$-th component of the Hom-complex is given by
$${\rm Hom}((x, y; a), (x', y'; a'))^p=\{\begin{pmatrix} \alpha & 0\\
                                                         h & \beta\end{pmatrix}\; |\; \alpha\in \mathcal{A}(x, x')^p,  \beta\in \mathcal{A}(y, y')^p, h\in \mathcal{A}(x, y')^{p-1}\},$$
with the differential given by
$$d \begin{pmatrix} \alpha & 0\\
h & \beta\end{pmatrix}= \begin{pmatrix} -d_\mathcal{A}(\alpha) & 0\\
 d_\mathcal{A}(h)+a'\circ \alpha-(-1)^p\beta \circ a & d_\mathcal{A}(\beta)\end{pmatrix}.$$                                                       The composition is given by the matrix multiplication.

Denote by ${\rm mor}(H^0(\mathcal{A}))$ the usual morphism category of the homotopy category $H^0(\mathcal{A})$ of $\mathcal{A}$.

\begin{lem}
The canonical functor
$$H^0({mor}(\mathcal{A})) \longrightarrow {\rm mor}(H^0(\mathcal{A})), \quad (x, y; a)\mapsto (x, y; \bar{a})$$
is full and dense, and its kernel ideal is square zero.
\end{lem}

\begin{proof}
For the fullness, just notice that a commutative square in $H^0(\mathcal{A})$ is given by a square in $Z^0(\mathcal{A})$,  which is commutative up to homotopy.

Take two morphisms in the kernel ideal: $\begin{pmatrix} d_\mathcal{A}(u) & 0\\
h & d_\mathcal{A}(v)\end{pmatrix}\colon(x, y; a) \rightarrow (x', y'; a')$ and $\begin{pmatrix} d_\mathcal{A}(u') & 0\\
h' & d_\mathcal{A}(v')\end{pmatrix}\colon(x', y'; a') \rightarrow (x'', y''; a'')$. Using $d_\mathcal{A}(h)=d_\mathcal{A}(v)\circ a-a'\circ d_\mathcal{A}(u)$ and $d_\mathcal{A}(h')=d_\mathcal{A}(v')\circ a'-a''\circ d_\mathcal{A}(u')$, we have
$$ \begin{pmatrix} d_\mathcal{A}(u') & 0\\
h' & d_\mathcal{A}(v')\end{pmatrix} \circ \begin{pmatrix} d_\mathcal{A}(u) & 0\\
h & d_\mathcal{A}(v)\end{pmatrix}=d\begin{pmatrix} -d_\mathcal{A}(u')\circ u & 0\\
                                        \tilde{h} & v'\circ d_\mathcal{A}(v)\end{pmatrix},$$
                                        where $\tilde{h}=v'\circ h-h'\circ u+ v'\circ a'\circ u$.
\end{proof}

We mention that the morphism dg category might be viewed as a certain gluing  \cite[Section 4]{KL} of dg categories.

\section{The category of dg modules}

In this section, the dg category $\mathcal{A}$ is required to be  small.

\subsection{DG modules}

\begin{defn}
A \emph{left dg $\mathcal{A}$-module} is a dg functor $M\colon \mathcal{A}\rightarrow C_{\rm dg}(k)$. The dg category of left dg $\mathcal{A}$-modules is given by $\mathcal{A}\mbox{-}{\rm DGMod}=\mathcal{H}{om}(\mathcal{A}, C_{\rm dg}(k)).$ The Hom-complex between two dg modules $M$ and $M'$ is denoted by ${\rm Hom}_\mathcal{A}(M, M')$.

By definition, a \emph{right dg $\mathcal{A}$-module} is a left dg $\mathcal{A}^{\rm op}$-module. Then we have the dg category ${\rm DGMod}\mbox{-}\mathcal{A}$ of right dg $\mathcal{A}$-modules.
\end{defn}

By default, modules are left modules. Later, we will justify the following viewpoint.

$${\boxed{\bf \mbox{dg modules over dg categories are better viewed as modules than functors!}}}$$

\begin{rem}
A left dg $\mathcal{A}$-module  $M\colon \mathcal{A}\rightarrow C_{\rm dg}(k)$ is identified with  the formal sum $\bigoplus_{x\in {\rm obj}(\mathcal{A})} M(x)$ of complexes. Here, by a \emph{formal sum}, we really mean a complex graded by ${\rm obj}(\mathcal{A})$ rather than a coproduct of complexes.  The \emph{left $\mathcal{A}$-action} on $M$ is defined as follows: for $a\colon x\rightarrow y$ and $m\in M(x)$,
$$a.m=M(a)(m)\in M(y).$$
Recall that $M(d_\mathcal{A}(a))=d(M(a)),$ where $d$ denotes the differential in $C_{\rm dg}(k)$. It is a nice exercise to prove that it is equivalent to the Leibniz rule
$$d_M(a.m)=d_\mathcal{A}(a).m+(-1)^{|a|} a. d_M(m).$$
We use the notation $_\mathcal{A} M$ for the left dg module.

Similarly, a right dg $\mathcal{A}$-module $N\colon \mathcal{A}^{\rm op}\rightarrow C_{\rm dg}(k)$ is identified with the formal sum $\bigoplus_{x\in {\rm obj}(\mathcal{A})} N(x)$. We define the \emph{right $\mathcal{A}$-action} on $N$ by $$n.a=(-1)^{|n|\cdot |a|} N(a)(n)\in N(x)$$
for $n\in N(y)$. The Leibniz rule for $N$ has the following form
$$d_N(n.a)=d_N(n).a +(-1)^{|n|} n.d_\mathcal{A}(a).$$
We use the notation $N_\mathcal{A}$ for the right dg module.
\end{rem}

\begin{exm}
For each object $y$, the  free dg  $\mathcal{A}$-module $\mathcal{A}(y, -)$ is identified with $\bigoplus_{x\in {\rm obj}(\mathcal{A})} \mathcal{A}(y, x)$, whose left $\mathcal{A}$-action is given by the composition of morphisms. We observe that in general, $\mathcal{A}(y, -)$ is not a projective object in the abelian category $Z^0(\mathcal{A}\mbox{-}{\rm DGMod})$.

For each $_\mathcal{A}M$, we have a canonical isomorphism of complexes
$${\rm Hom}_\mathcal{A}(\mathcal{A}(y, -), M) \longrightarrow M(y),\quad  f\mapsto f_y(1_y)\in M(y). $$
Consequently, the Yoneda dg functor
$$\mathcal{A}^{\rm op}\longrightarrow \mathcal{A}\mbox{-}{\rm DGMod}, \quad y\mapsto \mathcal{A}(y, -).$$
is fully faithful. A nice exercise: the Yoneda functor is a dg functor!
\end{exm}

\subsection{The suspension and cones}

As we will see, the dg category $\mathcal{A}\mbox{-}{\rm DGMod}$ has rich structures: the suspension dg functor and functorial cones.

\vskip 5pt

For a complex $X$, we denote by $\Sigma(X)$ the suspended complex: $\Sigma(X)^p=X^{p+1}$ and $d_{\Sigma(X)}=-d_X$. For each element $x\in X^{p+1}$, the corresponding element in $\Sigma(X)^p$ will be denoted by $\Sigma(x)$, meaning that $|\Sigma(x)|=|x|-1$.
\vskip 5pt

WARNING: $\Sigma(x)$ is just  a symbol to remember the change of degrees, and is not induced by any action on $X$!

\vskip 5pt

For $_\mathcal{A} M$, we have the suspended dg module $\Sigma(M)$ defined as follows: as complexes, $\Sigma(M)(x)=\Sigma(M(x))$ for each object $x$, and the new $\mathcal{A}$-action is given by
 $$a.\Sigma(m)=(-1)^{|a|} \Sigma(a.m)$$
  for $m\in M(x)$ and $a\in \mathcal{A}(x, y)$. For each morphism $f\colon M\rightarrow M'$ between dg modules, we set $\Sigma(f)\colon \Sigma(M)\rightarrow \Sigma(N)$ by $\Sigma(f)_x=(-1)^{|f|} f_x$.  Then we have the dg endofunctor
  $$\Sigma \colon \mathcal{A}\mbox{-}{\rm DGMod}\longrightarrow \mathcal{A}\mbox{-}{\rm DGMod},$$
  called the \emph{suspension dg functor}.

  For a closed morphism $f\colon M\rightarrow M'$ of degree zero, its \emph{mapping cone} ${\rm Cone}(f)$ is the dg $\mathcal{A}$-module defined as follows: as a graded $\mathcal{A}$-module, we have ${\rm Cone}(f)=M'\oplus \Sigma(M)$; the differential is given such that ${\rm Cone}(f)(x)={\rm Cone}(f_x)$. In other words,
  $$d_{{\rm Cone}(f)}=\begin{pmatrix} d_{M'} & f\\
                                    0 & d_{\Sigma(M)}\end{pmatrix}.$$
                                    Here, we implicitly use the identification $\Sigma(M)$ with $M$, and  obtain $f\colon \Sigma(M)=M\rightarrow M'$. The identification $\Sigma(M)\rightarrow M$ is actually a closed isomorphism of degree one. This will be clearer in the general consideration later.

                                    The following diagram
                                    \begin{align}\label{equ:cone}
                                    \xymatrix{ M'\ar[rr]^-{\binom{1}{0}} && {\rm Cone}(f) \ar@/^1pc/@{.>}[ll]^-{(1,0)} \ar[rr]^-{(0, 1)} &&  \Sigma(M) \ar@/^1pc/@{.>}[ll]^-{\binom{0}{1}} }
                                    \end{align}
                                    is a biproduct in the category of graded $\mathcal{A}$-modules. The solid arrows are closed morphisms, and the dotted arrows in general are not closed.

\subsection{DG bimodules}

\begin{defn}
A \emph{dg $\mathcal{A}$-$\mathcal{B}$-bimodule} is a dg functor $X\colon \mathcal{A}\otimes \mathcal{B}^{\rm op}\rightarrow C_{\rm dg}(k)$. The complex $X(u\otimes y)$ is usually denoted by $X(y, u)$ for $u\in {\rm obj}(\mathcal{A})$ and $y\in {\rm obj}(\mathcal{B})$. Consequently, $X(y, -)$ is a left dg $\mathcal{A}$-module and $X(-, u)$ is a right dg $\mathcal{B}$-module.
\end{defn}

\begin{rem}
The dg bimodule $_\mathcal{A}X_\mathcal{B}$ might be viewed as a bunch of left dg $\mathcal{A}$-modules indexed by ${\rm obj}(\mathcal{B})$, and in the same time a bunch of right dg $\mathcal{B}$-modules indexed by ${\rm obj}(\mathcal{A})$. Strictly speaking, the notations $_\mathcal{A}X$ and $X_\mathcal{B}$ are not well-defined.
\end{rem}

\begin{exm}
For a dg functor $F\colon \mathcal{A}\rightarrow \mathcal{B}$, we have the dg $\mathcal{A}$-$\mathcal{B}$-bimodule $_F\mathcal{B}_1$ given by: ${_F\mathcal{B}_1}(y, u)=\mathcal{B}(y, F(u))$. Similarly, we have the dg $\mathcal{B}$-$\mathcal{A}$-bimodule $_1\mathcal{B}_F$.
\end{exm}

For a dg $\mathcal{A}$-$\mathcal{B}$-bimodule $X$, we have the Hom  dg functor
$${\rm Hom}_\mathcal{A}(X, -)\colon \mathcal{A}\mbox{-}{\rm DGMod} \longrightarrow \mathcal{B}\mbox{-}{\rm DGMod}.$$
The left dg $\mathcal{B}$-module ${\rm Hom}_\mathcal{A}(X, M)$ is described as follows: for $y\in {\rm obj}(\mathcal{B})$, we have ${\rm Hom}_\mathcal{A}(X, M)(y)={\rm Hom}_\mathcal{A}(X(y, -), M)$, so, it is identified with the formal sum $\bigoplus_{y\in {\rm obj}(\mathcal{B})} {\rm Hom}_\mathcal{A}(X(y, -), M)$;  the left $\mathcal{B}$-action on ${\rm Hom}_\mathcal{A}(X, M)$ is induced by the right $\mathcal{B}$-action on $X$,  that is, for a morphism $b\colon y\rightarrow y'$ in $\mathcal{B}$, $f\colon X(y, -)\rightarrow M$ and $x\in X(y', -)$, we have
$$(b.f)(x)=(-1)^{|b|\cdot (|f|+|x|)}f(x.b).$$

 The tensor dg functor is given by
$$X\otimes_\mathcal{B}-\colon \mathcal{B}\mbox{-}{\rm DGMod} \longrightarrow \mathcal{A}\mbox{-}{\rm DGMod}.$$
Here, the tensor product $X\otimes_\mathcal{B} N$ is identified with $\bigoplus_{u\in {\rm obj}(\mathcal{A})} X(-, u)\otimes_\mathcal{B} N$. The complex  $ X(-, u)\otimes_\mathcal{B} N$ is defined by
$$(\oplus_{y\in {\rm obj}(\mathcal{B})}X(y, u)\otimes N(y))/{\langle x.b\otimes n-x\otimes b.n \rangle}.$$
We observe  a dg isomorphism
$$X\otimes_\mathcal{B} \mathcal{B}(y, -)\simeq X(y, -)$$
 of left $\mathcal{A}$-modules.

\begin{lem}
There is a canonical isomorphism of complexes
$${\rm Hom}_\mathcal{A}(X\otimes_\mathcal{B} N, M)\longrightarrow {\rm Hom}_\mathcal{B}(N, {\rm Hom}_\mathcal{A}(X, M)).$$
\end{lem}

\section{The derived category}
The standard reference for the derived category of a dg category is \cite{Ke94}.

\subsection{The definitions and resolutions}

The ordinary category $Z^0(\mathcal{A}\mbox{-}{\rm DGMod})$ is abelian. Here, we observe that a closed morphism of degree zero $f\colon M\rightarrow M'$ necessarily  respects the grading, differentials and $\mathcal{A}$-actions. An exact sequence
$$0\longrightarrow X \stackrel{f}\longrightarrow Y \stackrel{g}\longrightarrow Z\longrightarrow 0$$
in $Z^0(\mathcal{A}\mbox{-}{\rm DGMod})$ is \emph{graded-split}, provided that there exists $t\in {\rm Hom}_\mathcal{A}(Z, Y)^0$ satisfying $g\circ t=1_Z$, or equivalently, there exists $s\in {\rm Hom}_\mathcal{A}(Y, X)^0$ satisfying $s\circ f=1_X$.

\begin{prop}
The category $Z^0(\mathcal{A}\mbox{-}{\rm DGMod})$ is a Frobenius exact category, whose conflations are given by graded-split short exact sequences. The resulting stable category coincides with $H^0(\mathcal{A}\mbox{-}{\rm DGMod})$.
\end{prop}

\begin{proof}
The ingredient is the canonical adjoint triple between  $Z^0(\mathcal{A}\mbox{-}{\rm DGMod})$ and the abelian category of graded $\mathcal{A}$-modules; see \cite[Lemma 2.2]{Ke94}.
\end{proof}

We will write $\mathbf{K}(\mathcal{A})=H^0(\mathcal{A}\mbox{-}{\rm DGMod})$, the \emph{homotopy category} of dg $\mathcal{A}$-modules. A dg $\mathcal{A}$-module $M$ is \emph{acyclic} if each complex $M(x)$ is acyclic. These dg modules form  a triangulated subcategory $\mathbf{K}^{\rm ac}(\mathcal{A})$.

\begin{defn}
The \emph{derived category} of left dg $\mathcal{A}$-modules
   $$\mathbf{D}(\mathcal{A})=\mathbf{K}(\mathcal{A})/{\mathbf{K}^{\rm ac}(\mathcal{A})}.$$
\end{defn}

It turns out that $\mathbf{D}(\mathcal{A})$ is compactly generated by the  modules $\mathcal{A}(x, -)$.  The \emph{perfect derived category} of dg $\mathcal{A}$-modules is by definition the subcategory $\mathbf{perf}(\mathcal{A})=\mathbf{D}(\mathcal{A})^c$ of compact objects.

The following  is a fundamental theorem due to Keller; see \cite[Theorem 4.3]{Ke94}.

\begin{thm}\label{thm:1}
Each algebraic compactly generated triangulated category $\mathcal{T}$ is of the form $\mathbf{D}(\mathcal{A})$.
\end{thm}

\begin{proof}
Lift the subcategory $\mathcal{T}^c$ of compact objects to a dg category!
\end{proof}

To actually work with $\mathbf{D}(\mathcal{A})$, we need dg-projective resolutions and dg-injective resolutions.

A dg $\mathcal{A}$-module $P$ is \emph{dg-projective}, if for each surjective quasi-isomorphism $\pi\colon X\rightarrow Y$  and every morphism $f\colon P\rightarrow Y$ both in $Z^0(\mathcal{A}\mbox{-DGMod})$, there is a lifting $\tilde{f}\colon P\rightarrow X$, also in $Z^0(\mathcal{A}\mbox{-DGMod})$.

\begin{lem}
A dg $\mathcal{A}$-module $P$ is dg-projective if and only if $P$ is projective as a graded $\mathcal{A}$-module and ${\rm Hom}_{\mathbf{K}(\mathcal{A})}(P, N)=0$ for any acyclic module $N$.
\end{lem}

The dg-projective resolution is as follows.

\begin{lem}
Each dg module $X$ admits a surjective quasi-isomorphism
$${\bf p}(X)\longrightarrow X$$
with ${\bf p}(X)$ dg-projective, which is unique up to a unique isomorphism in $\mathbf{K}(\mathcal{A})$.
\end{lem}

Denote by $\mathbf{K}_{\rm p}(\mathcal{A})$ the homotopy category of dg-projective modules. Then the dg-projective resolution yields a triangle equivalence
$${\bf p}\colon \mathbf{D}(\mathcal{A})\longrightarrow \mathbf{K}_{\rm p}(\mathcal{A}).$$
Dually, we have the \emph{dg-injective resolution} $X \rightarrow {\bf i}(X)$ and the equivalence
 $${\bf i}\colon \mathbf{D}(\mathcal{A})\longrightarrow \mathbf{K}_{\rm i}(\mathcal{A}).$$
  To summarize, we have a recollement
\begin{align}
\xymatrix{
\mathbf{K}^{\rm ac}(\mathcal{A}) \ar[rr]|{{\rm inc}} &&  \mathbf{K}(\mathcal{A}) \ar[rr]|{\rm can}  \ar@/_1pc/[ll] \ar@/^1pc/[ll] &&   \mathbf{D}(\mathcal{A}).   \ar@/_1pc/[ll]|{\bf p} \ar@/^1pc/[ll]|{\bf i}
}
\end{align}

\subsection{Derived functors and recollements}

We will discuss derived functors.

$${\boxed{\bf \mbox{resolutions make life easier, but not really easier!}}}$$

\vskip 10pt

For a dg $\mathcal{A}$-$\mathcal{B}$-bimodule $X$, we have
$$X\otimes_\mathcal{B}^\mathbb{L}-\colon \mathbf{D}(\mathcal{B})\stackrel{\bf p}\longrightarrow \mathbf{K}(\mathcal{B}) \xrightarrow{X\otimes_\mathcal{B}-} \mathbf{K}(\mathcal{A})\stackrel{\rm can}\longrightarrow \mathbf{D}(\mathcal{A})$$
and
$$\mathbb{R}{\rm Hom}_\mathcal{A}(X, -)\colon \mathbf{D}(\mathcal{A})\stackrel{\bf i}\longrightarrow \mathbf{K}(\mathcal{A}) \xrightarrow{{\rm Hom}_\mathcal{A}(X, -)} \mathbf{K}(\mathcal{B})\stackrel{\rm can}\longrightarrow \mathbf{D}(\mathcal{B}).$$
They form an adjoint pair $(X\otimes_\mathcal{B}^\mathbb{L}-, \mathbb{R}{\rm Hom}_\mathcal{A}(X, -))$.

\begin{exm}
Let $F\colon \mathcal{B}\rightarrow \mathcal{A}$ be a dg functor. Then we have an adjoint triple
\[\xymatrix{
 \mathbf{D}(\mathcal{A}) \ar[rr] &&   \mathbf{D}(\mathcal{B}),   \ar@/_1pc/[ll]|{\bf F_*} \ar@/^1pc/[ll]|{\bf F_!}
}\]
where $F_*={_1\mathcal{A}_F}\otimes^\mathbb{L}_\mathcal{B}-$, $F_!=\mathbb{R}{\rm Hom}_\mathcal{B}({_F\mathcal{A}_1}, -)$ and the unnamed arrow sends $M$ to the composition $MF$. We observe that it is isomorphic to both $\mathbb{R}{\rm Hom}_\mathcal{A}({_1\mathcal{A}}_F, -)$ and ${_F\mathcal{A}_1}\otimes^\mathbb{L}_\mathcal{A}-$.
\end{exm}

\begin{prop}\label{prop:rec}
Assume that $j\colon \mathcal{B}\rightarrow \mathcal{A}$ is a quasi-fully faithful dg functor. Then there is a recollement
\begin{align}\xymatrix{
\mathcal{K} \ar[rr]|{{\rm inc}} &&  \mathbf{D}(\mathcal{A}) \ar[rr]|{\rm res}  \ar@/_1pc/[ll] \ar@/^1pc/[ll] &&   \mathbf{D}(\mathcal{B}).   \ar@/_1pc/[ll]|{j_*} \ar@/^1pc/[ll]|{j_!}
}
\end{align}
Moreover, there is a dg category $\mathcal{C}$ such that $\mathcal{K}$ is triangle equivalent to $\mathbf{D}(\mathcal{C})$. Consequently, we have a triangle equivalence up to direct summands
$$\mathbf{perf}(\mathcal{A})/{\mathbf{perf}(\mathcal{B})} \longrightarrow \mathbf{perf}(\mathcal{C}).$$
\end{prop}

\begin{proof}
We observe that $j_*$ is fully faithful, since it extends $H^0(j)$. The last statement follows from Theorem \ref{thm:1}.
\end{proof}

The following problem will be central for us:

$${\boxed{\bf \mbox{what is the above dg category $\mathcal{C}$?}}}$$

\section{The dg quotient category}

In this section, we recall the explicit construction of Drinfeld on the dg quotient category; see \cite[Section 3]{Dri}. It solves the above problem in an elegant way.

\subsection{The construction of Drinfeld}

Let $\mathcal{B}\subseteq \mathcal{A}$ be a full dg subcategory. Denote by $j\colon \mathcal{B}\rightarrow \mathcal{A}$  the inclusion.

\begin{defn}\label{defn:dgq}
The \emph{dg quotient category} $\mathcal{A}/\mathcal{B}$ is defined as follows:
\begin{enumerate}
\item[$\bullet$] ${\rm obj}(\mathcal{A}/\mathcal{B})={\rm obj}(\mathcal{A})$;
\item[$\bullet$] freely add new morphisms $\varepsilon_U\colon U\rightarrow U$ of degree $-1$ for each $U\in {\rm obj}(\mathcal{B})$, and set $d(\varepsilon_U)=1_U$.
\end{enumerate}
\end{defn}

Denote by $\pi\colon \mathcal{A}\rightarrow \mathcal{A}/\mathcal{B}$ the canonical functor.

\begin{rem}
For any objects $x$ and $y$, we have
$$\mathcal{A}/\mathcal{B}(x, y)=\bigoplus_{n\geq 0} \bigoplus_{U_i\in {\rm obj}(\mathcal{B})} \mathcal{A}(U_n, y)\otimes k\varepsilon_{U_n}\otimes \cdots \otimes k\varepsilon_{U_2}\otimes \mathcal{A}(U_1, U_2)\otimes k\varepsilon_{U_1}\otimes \mathcal{A}(x, U_1).$$
This is only a decomposition of graded $k$-modules, not of complexes. Indeed, it gives rise to an ascending filtration
$$\mathcal{A}(x, y)={\mathcal{A}/\mathcal{B}}^0(x, y)\subseteq {\mathcal{A}/\mathcal{B}}^{\leq 1}(x, y)\subseteq {\mathcal{A}/\mathcal{B}}^{\leq 2}(x, y)\subseteq \cdots$$
of subcomplexes, whose factors are denoted by ${\mathcal{A}/\mathcal{B}}^n(x, y)$ or ${\rm Hom}_{\mathcal{A}/\mathcal{B}}^n(x, y)$. We observe that each object $U$ in $\mathcal{B}$ becomes contractible in $\mathcal{A}/\mathcal{B}$.
\end{rem}

Under mild assumptions, the above dg quotient category describes the dg category $\mathcal{C}$ in Proposition \ref{prop:rec}; see \cite[Proposition 4.6(ii)]{Dri}.

\begin{thm}\label{thm:2}
Assume that $\mathcal{A}(x, U)$ is a dg-projective $k$-module for each $x\in {\rm obj}(\mathcal{A})$ and $U\in {\rm obj}(\mathcal{B})$. Then there is a recollement
\begin{align}\xymatrix{
\mathbf{D}(\mathcal{A}/\mathcal{B}) \ar[rr]|{{\rm can}}  &&  \mathbf{D}(\mathcal{A}) \ar[rr]|{\rm res}  \ar@/_1pc/[ll]|{\pi_*} \ar@/^1pc/[ll]|{\pi_!} &&   \mathbf{D}(\mathcal{B}),   \ar@/_1pc/[ll]|{j_*} \ar@/^1pc/[ll]|{j_!}
}
\end{align}
where ``${\rm can}$" sends $M$ to the composition $M\pi$. Consequently, we have a triangle equivalence up to direct summands
$$\mathbf{perf}(\mathcal{A})/{\mathbf{perf}(\mathcal{B})} \longrightarrow \mathbf{perf}(\mathcal{A}/\mathcal{B}).$$
 \end{thm}

 \begin{proof}
 It suffices to prove that ${\rm Im}\; {\rm can}={\rm Ker}\; {\rm res}$ and that ``${\rm can}$" is fully faithful.

 We observe ${\rm Ker}\; {\rm res}=\{_\mathcal{A}X\; |\; \mbox{the restriction }X|_{\mathcal{B}} \mbox{ is acyclic}\}$. Then ${\rm Im}\; {\rm can}\subseteq{\rm Ker}\; {\rm res}$. Conversely, let $_\mathcal{A}X\in {\rm Ker}\; {\rm res}$ be dg-projective. The assumption implies that the complex $X(U)$ of $k$-modules is dg-projective, and thus contractible. Take a contracting homotopy $s_U$ of degree $-1$ for each $U\in {\rm obj}(\mathcal{B})$. Using the universal property, we lift $X$ to a dg module over $\mathcal{A}/\mathcal{B}$.

 For the fully-faithfulness of ``${\rm can}$", we observe that ${\rm can}(f)$ is an isomorphism if and only if so is $f$.  We apply Lemma \ref{lem:cone} to each $X\in {\rm Ker}\; {\rm res}$. It follows that ${\rm Cone}(\theta_X)\in {\rm Im}\;j_*\cap {\rm Ker}\; {\rm res}=0$. So, $\theta_X$ is an isomorphism. This yields another proof of the above equality.

 Now, for each dg $\mathcal{A}/\mathcal{B}$-module $Y$, the counit $\delta_Y\colon \pi_*{\rm can}(Y)\rightarrow Y$ has to be an isomorphism, since ${\rm can}(\delta_Y)$ is an isomorphism.
 \end{proof}

 The key observation is as follows.

 \begin{lem}\label{lem:cone}
 For each $_\mathcal{A}X$, the cone of the unit $\theta_X\colon X\rightarrow {\rm can}\pi_*(X)$ lies in ${\rm Im}\; j_*$.
 \end{lem}

 \begin{proof}
 The general case follows from the case where $X=\mathcal{A}(x, -)$. We observe ${\rm can}\pi_*(X)={\mathcal{A}/\mathcal{B}}(x, -)$. Then the statement follows from the following exact sequence
 in $Z^0(\mathcal{A}\mbox{-DGMod})$
 $$0\longrightarrow \mathcal{A}(x, -) \stackrel{\rm inc}\longrightarrow {\mathcal{A}/\mathcal{B}}(x, -) \longrightarrow {\mathcal{A}/\mathcal{B}}^{\geq 1}(x, -)\longrightarrow 0,$$
 where the rightmost term is generated by $\mathcal{A}(U, -)$'s, using the assumption.
 \end{proof}

 \begin{rem}
 In view of \cite[Theorem 4.11 (iii)]{Ke06}, there is an exact sequence
 $$\mathcal{B} \stackrel{j}\longrightarrow \mathcal{A} \stackrel{\pi}\longrightarrow \mathcal{A}/\mathcal{B}$$
 in $\mathbf{Hmo}$, the homotopy category of small dg categories with respect to the Morita-model structures. We mention that $\mathbf{Hmo}$ is equivalent to a full subcategory of $\mathbf{Hodgcat}$.
 \end{rem}

\subsection{An example}

We assume that $k$ is a field. Let $A$ be an ordinary $k$-algebra. Fix a set $\{P_i\; |\; i\in \Lambda\}$ of two-term complexes of finitely generated projective right $A$-modules. Denote by $\pi\colon A\rightarrow B$ the \emph{universal localization}.

The complexes $\{P_i\; |\; i\in \Lambda\}$ form a full dg subcategory $\mathcal{B}$ of $C^b_{\rm dg}(A^{\rm op}\mbox{-proj})$. Denote by $\Gamma$ the dg endomorphism algebra of $A_A$ in $C^b_{\rm dg}(A^{\rm op}\mbox{-proj})/\mathcal{B}$.

\begin{prop}
There is a recollement
\[\xymatrix{
\mathbf{D}(\Gamma) \ar[rr] &&  \mathbf{D}(A\mbox{-{\rm Mod}}) \ar[rr]  \ar@/_1pc/[ll] \ar@/^1pc/[ll] &&   \mathbf{D}(\mathcal{B}),   \ar@/_1pc/[ll]  \ar@/^1pc/[ll]
}\]
Consequently, $\Gamma$ is non-positively graded with $H^0(\Gamma)\simeq B$. Furthermore, $\pi$ is a homological epimorphism if and only if $\Gamma$ is quasi-isomorphic to $B$.
\end{prop}

\begin{proof}
In view of Theorem \ref{thm:2}, it suffices to  identify $\mathbf{D}(A)=\mathbf{D}(A\mbox{-{\rm Mod}})$ with $\mathbf{D}(C^b_{\rm dg}(A^{\rm op}\mbox{-proj}))$, $\mathbf{D}(\Gamma)$ with $\mathbf{D}(C^b_{\rm dg}(A^{\rm op}\mbox{-proj})/\mathcal{B})$; see Lemma \ref{lem:mor-equ}.
\end{proof}

Let us consider a special case, where $\mathcal{B}$ is given by a single stalk complex $eA$  for some idempotent $e$ in $A$. Then the dg algebra $\Gamma$ is isomorphic to the tensor algebra
$$T_A(Ae\otimes k\varepsilon\otimes eA),$$ with $|\varepsilon|=-1$ and $d(e\otimes \varepsilon\otimes e)=e$. We observe that $H^0(\Gamma)=A/AeA$, $H^{-1}(\Gamma)={\rm Ker}(Ae\otimes_{eAe} eA\rightarrow A)$, and $H^{-p}(\Gamma)={\rm Tor}_{p-1}^{eAe}(Ae, eA)$ for $p\geq 2$. Then we recover the classical result: the ideal $AeA$ is stratifying if and only if the natural projection $A\rightarrow A/AeA$ is a homological epimorphism.

\section{Exact dg categories}

In this section, we justify the following fact: dg categories have functorial cones. Indeed,  an exact dg category has intrinsic suspensions and cones.

\subsection{The suspensions and cones in general}

The following notions are taken from \cite[Section 3]{BLL}.

\begin{defn}
For $x\in {\rm obj}(\mathcal{A})$, its \emph{suspension} means an object $x'$ together with a closed isomorphism $\xi_x\colon x'\rightarrow x$ of degree one. Notation: $x'=\Sigma(x)$, as it is unique up to a unique dg isomorphism.
\end{defn}

\begin{lem}\label{lem:sus}
An object $y$ is dg isomorphic to $\Sigma(x)$ if and only if $\mathcal{A}(-, y)$ is dg isomorphic to $\Sigma \mathcal{A}(-, x)$, as a right dg $\mathcal{A}$-module.
\end{lem}

\begin{proof}
$\mathcal{A}(-, y)\simeq \Sigma \mathcal{A}(-, x)$ if and only if there is a closed isomorphism $\mathcal{A}(-, y)\simeq \mathcal{A}(-, x)$ of degree one. Now, use the Yoneda embedding $\mathcal{A}\hookrightarrow {\rm DGMod}\mbox{-}\mathcal{A}$.
 \end{proof}

Assume that each object has a suspension. Then we have a fully faithful dg functor
$$\Sigma\colon \mathcal{A}\longrightarrow \mathcal{A}, \quad x\mapsto \Sigma(x).$$
For a morphism $a\colon x\rightarrow y$, we have $\Sigma(a)=(-1)^{|a|} (\xi_y)^{-1}\circ a \circ \xi_x$. We call $\Sigma$ the \emph{suspension dg functor}.

Let $f\colon x\rightarrow y$ be a morphism in $Z^0(\mathcal{A})$. Assume that $\xi_x\colon \Sigma(x)\rightarrow x$ is the chosen isomorphism.

\begin{defn}
The \emph{cone} of $f$ is defined to be an object ${\rm Cone}(f)$ together with a diagram
\[
\xymatrix{ y \ar[rr]^-{j} && {\rm Cone}(f) \ar@/^1pc/@{.>}[ll]^-{t} \ar[rr]^-{p} &&  \Sigma(x) \ar@/^1pc/@{.>}[ll]^-{s} }\]
such that
\begin{enumerate}
\item[(C1)] the diagram is a biproduct in $\mathcal{A}^0$, the underlying category with morphisms of degree zero;
\item[(C2)] $d(j)=0=d(p)$;
\item[(C3)] $f=t\circ d(s)\circ \xi_x^{-1}=-d(t)\circ s\circ \xi_x^{-1}$.
\end{enumerate}
\end{defn}

\begin{rem}
\begin{enumerate}
\item Assume (C1) and (C2). A nice exercise: prove that  (C3) is equivalent to any of the following identities: $d(s)=j\circ f\circ \xi_x$, $d(t)=-f\circ \xi_x\circ p$.
\item The cone is unique up to a unique dg isomorphism. We refer to \cite[Subsection~2.4, the second paragraph]{Ke11} for a characterization of ${\rm Cone}(f)$ using a pair of maps with a universal property.
\end{enumerate}
\end{rem}

\begin{lem}\label{lem:inter-cone}
An object $z$ is dg isomorphic to ${\rm Cone}(f)$ if and only if $\mathcal{A}(-, z)$ is dg isomorphic to ${\rm Cone}(\mathcal{A}(-, f)\colon \mathcal{A}(-, x)\rightarrow \mathcal{A}(-, y))$.
\end{lem}

\begin{proof}
We use (\ref{equ:cone}) and the Yoneda embedding $\mathcal{A}\hookrightarrow {\rm DGMod}\mbox{-}\mathcal{A}$.
\end{proof}

As promised, we prove that cones are functorial in a dg category.

\begin{prop}
Assume that each object in $\mathcal{A}$ has a suspension and each closed morphism of degree zero has a cone. Then we have the cone dg  functor
$${\rm Cone}\colon {mor}(\mathcal{A})\longrightarrow \mathcal{A}, \quad (x, y; f)\mapsto {\rm Cone}(f).$$
It sends a morphism $\begin{pmatrix}\alpha & 0\\
                                   h & \beta\end{pmatrix}\colon (x, y; f)\rightarrow (x', y'; f')$ to $$\begin{pmatrix} \beta & h\circ \xi_x\\
                                                   0 & (-1)^{|\alpha|} \Sigma(\alpha)\end{pmatrix}\colon {\rm Cone}(f)\longrightarrow {\rm Cone}(f'),$$ where we identify ${\rm Cone}(f)$ with $y\oplus \Sigma(x)$,  ${\rm Cone}(f')$ with $y'\oplus \Sigma(x')$.
\end{prop}

\begin{proof}
Just remember that $\begin{pmatrix} \beta & h\circ \xi_x\\
                                                   0 & (-1)^{|\alpha|} \Sigma(\alpha)\end{pmatrix}$ really means $(-1)^{|\alpha|} s'\circ \Sigma(\alpha)\circ p + j'\circ \beta \circ t + j'\circ h\circ \xi_x\circ p$.
\end{proof}

\begin{rem}
For a closed morphism $\begin{pmatrix}\alpha & 0\\
                                   h & \beta\end{pmatrix}$ of degree zero, we obtain a commutative diagram in $Z^0(\mathcal{A})$ as follows:
                                   \[\xymatrix{
                                   y\ar[rr]^-{j}\ar[d]_-{\beta} && {\rm Cone}(f) \ar[d]^-{ \tiny \begin{pmatrix} \beta & h\circ \xi_x\\
                                                   0 &  \Sigma(\alpha)\end{pmatrix}} \ar[rr]^-{p} && \Sigma(x) \ar[d]^-{\Sigma(\alpha)}\\
                                   y'\ar[rr]^-{j'} && {\rm Cone}(f) \ar[rr]^-{p'} && \Sigma(x).
                                   }\]
\end{rem}

\subsection{The definition}

The following notion, due to \cite[Section 2]{Ke99}, should be central.

\begin{defn}
A dg category $\mathcal{A}$ is \emph{exact} provided that
\begin{enumerate}
\item each object has a suspension, and $\Sigma\colon \mathcal{A}\rightarrow \mathcal{A}$ is dg dense;
\item each closed morphism of degree zero has a cone.
\end{enumerate}
\end{defn}

Exact dg categories are also called \emph{strongly pretriangulated} dg categories; see Remark~\ref{rem:confu} and compare Definition~\ref{defn:pret}.

\begin{lem}
The dg category $\mathcal{A}$ is exact if and only if the essential image of the Yoneda embedding $\mathcal{A}\hookrightarrow {\rm DGMod}\mbox{-}\mathcal{A}$ is closed under $\Sigma^{\pm 1}$ and cones.
\end{lem}

\begin{proof}
We use Lemmas \ref{lem:sus} and \ref{lem:inter-cone}.
\end{proof}

The following fact justifies the importance of exact dg categories; see \cite[Lemma~2.3]{Ke99}.

\begin{prop}
Let $\mathcal{A}$ be an exact dg category. Then $Z^0(\mathcal{A})$ has a canonical Frobenius exact structure, whose stable category coincides with $H^0(\mathcal{A})$. Therefore, $H^0(\mathcal{A})$ is canonically triangulated.

Furthermore, let $F\colon \mathcal{A}\rightarrow \mathcal{B}$ be a dg functor between two exact dg categories. Then $Z^0(F)\colon Z^0(\mathcal{A})\rightarrow Z^0(\mathcal{B})$ is an exact functor preserving projective objects. Consequently, $H^0(F)\colon H^0(\mathcal{A}) \rightarrow H^0(\mathcal{B})$ is naturally a triangle functor.
\end{prop}

\begin{proof}
We observe $Z^0(\mathcal{A})$ is a full subcategory of $Z^0({\rm DGMod}\mbox{-}\mathcal{A})$. Each  conflation in $Z^0(\mathcal{A})$ is given by a sequence $x\rightarrow e \rightarrow y$, which is a part of a biproduct in $\mathcal{A}^0$. We observe that $H^0(\mathcal{A})\subseteq \mathbf{perf}(\mathcal{A}^{\rm op})$ is a triangulated subcategory.

The second statement follows immediately, since everything is intrinsic.
\end{proof}

\begin{defn}
The \emph{pretriangulated  hull} $\mathcal{A}^{\rm pretr}$ is the smallest dg subcategory of ${\rm DGMod}\mbox{-}\mathcal{A}$ containing $\mathcal{A}$, closed under $\Sigma^{\pm}$ and cones. As $\mathcal{A}^{\rm pretr}$ is exact, we define the \emph{triangulated hull}  of $\mathcal{A}$ to be $\mathcal{A}^{\rm tr}=H^0(\mathcal{A}^{\rm pretr})$.
\end{defn}

\begin{rem}\label{rem:tr}
\begin{enumerate}
\item The dg category $\mathcal{A}$ is exact if and only if the canonical embedding ${\rm can}_\mathcal{A} \colon \mathcal{A}\rightarrow \mathcal{A}^{\rm pretr}$ is a dg equivalence.
    \item The inclusion $\mathcal{A}^{\rm tr}\subseteq \mathbf{perf}(\mathcal{A}^{\rm op})$ is a triangle equivalence up to direct summands.
    \item There is a concrete description of the pretriangulated hull; see \cite{BK} and \cite[Subsection 2.4]{Dri}.
    \end{enumerate}
\end{rem}

\begin{lem}
The canonical embedding ${\rm can}_\mathcal{A}\colon \mathcal{A}\rightarrow \mathcal{A}^{\rm pretr}$ has the following universal property: for any dg functor $F\colon \mathcal{A}\rightarrow \mathcal{B}$ into an exact dg category $\mathcal{B}$, there is a dg functor $\tilde{F}\colon \mathcal{A}^{\rm pretr}\rightarrow \mathcal{B}$ with a natural dg isomorphism $\theta\colon F\simeq \tilde{F}{\rm can}_\mathcal{A}$; moreover, the pair $(\tilde{F}, \theta)$ is unique in the obvious sense.

The dg functor $\tilde{F}$ is a dg equivalence if and only if $F$ is fully faithful and ${\rm Im}\; F$ generates $\mathcal{B}$ as an exact dg category.
\end{lem}

\begin{proof}
Set $\tilde{F}=({\rm can}_\mathcal{B})^{-1} F^{\rm pretr}$.
\end{proof}

\begin{rem}
The ``correct" statement for the universal property is as follows; see \cite[Subsection 4.5]{Ke06}: for any dg category $\mathcal{B}$, there is a dg equivalence
$$\mathcal{H}om(\mathcal{A}^{\rm pretr}, \mathcal{B}) \stackrel{\sim}\longrightarrow \mathcal{H}om(\mathcal{A}, \mathcal{B}), \quad G\mapsto G {\rm can}_\mathcal{A}.$$
\end{rem}

The following easy fact is useful.

\begin{lem}\label{lem:mor-equ}
Assume that $j\colon \mathcal{B}\rightarrow \mathcal{A}$ is a fully-faithful dg functor such that $j^{\rm pretr}$ is a dg equivalence. Then the restriction functor
$${\rm res}\colon\mathbf{D}(\mathcal{A})\longrightarrow \mathbf{D}(\mathcal{B})$$
is a triangle equivalence.
\end{lem}

\section{The dg quotient vs Verdier quotient}

In this section, we  recall a fundamental theorem \cite[Theorem~3.4]{Dri}, which shows that the dg quotient category enhances the classical Verdier quotient category.

\subsection{A triangle equivalence}

The following triangle equivalence shows that:

$${\boxed{\bf \mbox{dg quotient enhances Verdier quotient!}}}$$

\vskip 10pt

\begin{thm}\label{thm:3}
Let $\mathcal{B}\subseteq \mathcal{A}$ be a full dg subcategory. Assume that ${\rm Hom}_\mathcal{A}(x, U)$ is homotopically flat for each $x\in {\rm obj}(\mathcal{A})$ and $U\in {\rm obj}(\mathcal{B})$. Then the canonical functor
$$\Phi\colon \mathcal{A}^{\rm tr}/{\mathcal{B}^{\rm tr}}\stackrel{\sim}\longrightarrow (\mathcal{A}/\mathcal{B})^{\rm tr}$$
is a triangle equivalence.
\end{thm}

\begin{rem}
\begin{enumerate}
\item By duality, the equivalence also holds if  ${\rm Hom}_\mathcal{A}(U, x)$ is homotopically flat for each $x\in {\rm obj}(\mathcal{A})$ and $U\in {\rm obj}(\mathcal{B})$.
    \item Assume that ${\rm Hom}_\mathcal{A}(x, U)$ is dg-projective for each $x\in {\rm obj}(\mathcal{A})$ and $U\in {\rm obj}(\mathcal{B})$, or dually, ${\rm Hom}_\mathcal{A}(U, x)$ is dg-projective for each $x\in {\rm obj}(\mathcal{A})$ and $U\in {\rm obj}(\mathcal{B})$. Thanks to Remark \ref{rem:tr}(2), the above equivalence follows from the equivalence in Theorem \ref{thm:2}.
\end{enumerate}
\end{rem}

The following notions, due to \cite{BK}, are very convenient for the study of the homotopy category $\mathbf{Hodgcat}$. By definition, $\mathbf{Hodgcat}$ is the localization of $\mathbf{dgcat}$ with respect to quasi-equivalences.

\begin{defn}\label{defn:pret}
\begin{enumerate}
\item A dg category $\mathcal{A}$ is \emph{pretriangulated}, if the canonical embedding $\mathcal{A}\rightarrow \mathcal{A}^{\rm pretr}$ is a quasi-equivalence, or equivalently, $H^0(\mathcal{A})$ is a triangulated subcategory of $H^0({\rm DGMod}\mbox{-}\mathcal{A})$ via the Yoneda embedding.
    \item A \emph{dg enhancement} of a triangulated category $\mathcal{T}$ is a pretriangulated dg category $\mathcal{A}$ together with a triangle equivalence $E\colon \mathcal{T}\rightarrow H^0(\mathcal{A})$.
    \end{enumerate}
\end{defn}

\begin{rem}\label{rem:confu}
Let us try to clarify the confusion of terminologies in the literature. Exact dg categories are called $+$-pretriangulated in \cite[p.105, Remark]{BK}, strongly pretriangulated in \cite[p.650]{Dri} and \cite[p.1475]{BLL},  and pretriangulated in \cite[Subsection~4.5]{Ke06}. The common terminology in Definition~\ref{defn:pret}(1) is taken from \cite[p.650]{Dri} and \cite[p.1475]{BLL}.
\end{rem}

\begin{rem}
 \begin{enumerate}
 \item Let $F\colon \mathcal{A}\rightarrow \mathcal{B}$ be a quasi-equivalence. Then $\mathcal{A}$ is pretriangulated if and only if so is $\mathcal{B}$.
 \item A triangulated category has a dg enhancement if and only if it is algebraic. The uniqueness of dg enhancements is an active topic; see \cite{CS}.
     \end{enumerate}
\end{rem}

In practice, the following immediate consequence will be useful.

\begin{cor}\label{cor:quo}
Let $\mathcal{A}$ be a pretriangulated dg category and $\mathcal{B}\subseteq \mathcal{A}$ an pretriangulated dg full subcategory.  Assume that ${\rm Hom}_\mathcal{A}(x, U)$ is homotopically flat for each $x\in {\rm obj}(\mathcal{A})$ and $U\in {\rm obj}(\mathcal{B})$. Then $\mathcal{A}/\mathcal{B}$ is pretriangulated. Moreover, the canonical functor
$$H^0(\mathcal{A})/{H^0(\mathcal{B})}\stackrel{\sim}\longrightarrow H^0(\mathcal{A}/\mathcal{B})$$
is an isomorphism between triangulated categories.
\end{cor}

\begin{proof}
We identify $H^0(\mathcal{A})$ with $\mathcal{A}^{\rm tr}$, $H^0(\mathcal{B})$ with $\mathcal{B}^{\rm tr}$. For the isomorphism, we mention that the canonical functor acts on objects by the identity.
\end{proof}

As the bounded derived category of a module category is the central object in homological algebra, the following concept should be central in dg categories: it provides a canonical dg enhancement for the bounded derived category.

\begin{exm}
Let $\mathcal{E}$ be an exact category. Denote by $C^b_{\rm dg}(\mathcal{E})$ the dg category of bounded complexes in $\mathcal{E}$, and by $C^{b, {\rm ac}}_{\rm dg}(\mathcal{E})$ the dg subcategory formed by acyclic complexes. Then the dg quotient
$$\mathbf{D}^b_{\rm dg}(\mathcal{E})=C^b_{\rm dg}(\mathcal{E})/{C^{b, {\rm ac}}_{\rm dg}(\mathcal{E})}$$
is called the \emph{bounded dg derived category} of $\mathcal{E}$. By Corollary \ref{cor:quo}, under suitable conditions, the category $\mathbf{D}^b_{\rm dg}(\mathcal{E})$ is pretriangulated, and  we have an isomorphism of triangulated categories
$$\mathbf{D}^b(\mathcal{E})\stackrel{\sim}\longrightarrow H^0(\mathbf{D}^b_{\rm dg}(\mathcal{E})).$$
Therefore, the dg derived category canonically enhances the derived category.
\end{exm}

\subsection{A sketched proof}

We now sketch the proof. For details, we refer to \cite[Section 8]{Dri}.

\vskip 5pt

\noindent \emph{Proof of Theorem \ref{thm:3}.} \quad It suffices to show that $\Phi$ is fully faithful. Furthermore, it suffices to show that the natural map
$${\rm Ext}^i_{\mathcal{A}^{\rm tr}/\mathcal{B}^{\rm tr}}(x, y) \longrightarrow {\rm Ext}^i_{(\mathcal{A}/\mathcal{B})^{\rm tr}}(x, y)$$
is an isomorphism for $i\in \mathbb{Z}$, and $x, y\in {\rm obj}(\mathcal{A})$.

Set $Q_y$ to be the filtered category of $\mathcal{A}^{\rm tr}$-morphisms $\alpha\colon y\rightarrow z$ with ${\rm Cone}(\alpha)\in \mathcal{B}^{\rm tr}$. Then
\begin{align*}
LHS &= {\rm colim}_{\alpha\in Q_y} \; {\rm Ext}_{\mathcal{A}^{\rm tr}}^i (x, z)\\
    &= {\rm colim}_{\alpha\in Q_y} \; H^i {\rm Hom}_{\mathcal{A}^{\rm pretr}}(x, z).
\end{align*}

WARNING: Here, $H^i$ and ${\rm colim}$ can not commute, since $Q_y$ is a not filtered category of morphisms in $Z^0(\mathcal{A}^{\rm pretr})$!

\begin{align*}
RHS &= H^i{\rm Hom}_{\mathcal{A}/\mathcal{B}} (x, y)\\
    &= H^i {\rm Hom}_{\mathcal{A}^{\rm pretr}/\mathcal{B}} (x, y).
\end{align*}

For every $\alpha\colon y\rightarrow z\in Q_y$, we assume $\alpha=\bar{\alpha'}$, $\alpha'\in Z^0(\mathcal{A}^{\rm pretr})$. So, ${\rm Cone}(\alpha')$ lies in $\mathcal{B}^{\rm pretr}$ up to homotopy. We have a quasi-isomorphism
$${\rm Hom}_{\mathcal{A}^{\rm pretr}/\mathcal{B}}(x, y) \longrightarrow {\rm Hom}_{\mathcal{A}^{\rm pretr}/\mathcal{B}}(x, z),$$
as ${\rm Hom}_{\mathcal{A}^{\rm pretr}/\mathcal{B}}(x, {\rm Cone}(\alpha'))$ is acyclic (up to homotopy, ${\rm Cone}(\alpha')$ is an iterated extension of objects in  $\mathcal{B}$;  each $U\in {\rm obj}(\mathcal{B})$ is contractible in $\mathcal{A}^{\rm pretr}/\mathcal{B}$!).

So, we have
$$RHS={\rm colim}_{\alpha\in Q_y}\;  H^i {\rm Hom}_{\mathcal{A}^{\rm pretr}/\mathcal{B}}(x, z).$$

 Recall the following  exact sequence
$$0\longrightarrow {\rm Hom}_{\mathcal{A}^{\rm pretr}}(x, z)  \longrightarrow {\rm Hom}_{\mathcal{A}^{\rm pretr}/\mathcal{B}}(x, z) \longrightarrow  {\rm Hom}^{\geq 1}_{\mathcal{A}^{\rm pretr}/\mathcal{B}}(x, z) \longrightarrow 0.$$
Taking $H^i$ and then ${\rm colim}_{\alpha\in Q_y}$, we will be done by proving
$${\rm colim}_{\alpha\in Q_y}\;  H^i {\rm Hom}^{\geq 1}_{\mathcal{A}^{\rm pretr}/\mathcal{B}}(x, z)=0.$$

Using the filtration, the above equality follows from the following claim
$${\rm colim}_{\alpha\in Q_y}\;  H^i {\rm Hom}^n_{\mathcal{A}^{\rm pretr}/\mathcal{B}}(x, z)=0$$
for each $n\geq 1$. We observe that
$${\rm Hom}^n_{\mathcal{A}^{\rm pretr}/\mathcal{B}}(x, z)=\bigoplus_{U\in {\rm obj}(\mathcal{B})} {\rm Hom}_{\mathcal{A}^{\rm pretr}}(U, z)\otimes F_{x, U},$$
where the complex $F_{x, U}$ is homotopically flat.
But, we do have
$${\rm colim}_{\alpha\in Q_y} \; H^i{\rm Hom}_{\mathcal{A}^{\rm pretr}}(U, z)={\rm Ext}^i_{\mathcal{A}^{\rm tr}/{\mathcal{B}^{\rm tr}}}(U, y)=0.$$
Then the claim follows from Lemma \ref{lem:hf}, a general fact. \hfill $\square$

\vskip 5pt

\begin{lem}\label{lem:hf}
Let $\{C_\alpha\}_{\alpha\in \Lambda}$ be a filtered system of objects in $\mathbf{K}(k\mbox{-}{\rm Mod})$. Assume that ${\rm colim}_{\alpha\in \Lambda} H^i(C_\alpha)=0$ for each $i\in \mathbb{Z}$, and that $F$ is a homotopically flat complex. Then ${\rm colim}_{\alpha\in \Lambda} H^i(C_\alpha\otimes F)=0$ for each $i\in \mathbb{Z}$.
\end{lem}

\begin{rem}
The assumptions on the Hom-complexes in Theorems \ref{thm:2} and \ref{thm:3} might be ``removed". Indeed, it suffices to take a cofibrant replacement of $\mathcal{A}$. Similarly, to define the dg quotient category in the most general case, we have to replace $\mathcal{A}$ by its cofibrant replacement and then freely add the contracting homotopies as in Definition \ref{defn:dgq}.
\end{rem}

\printindex

\end{document}